\documentclass{birkmult}
\usepackage{amsfonts}
\usepackage{graphicx}
\usepackage{amsmath}
\usepackage{amssymb}%
\input{amssym.def}
\newtheorem{theorem}{Theorem}[section]
\newtheorem{proposition}{Proposition}[section]

\newtheorem{lemma}{Lemma}[section]
\newtheorem{remark}{Remark}[section]

\theoremstyle{remark}

\newcommand{\im}{\text{\rm Im }}

\begin{document}
%
%
%
%
%
%
%
%
%
\title[Singular and tangent slit solutions...]{Singular and tangent slit solutions to the
L\"owner equation}
\author{Dmitri Prokhorov}

\address{%
Department of Mathematics and Mechanics,\\
Saratov State University,\\
Astrakhanskay Str. 83,\\
410012, Saratov\\
Russia}

\email{prokhorov@sgu.ru}

\author{Alexander Vasil'ev}

\address{%
Department of Mathematics,\\ University of Bergen, \\Johannes Brunsgate 12,
Bergen 5008,\\ Norway}

\email{alexander.vasiliev@math.uib.no}

\thanks{ The first author was partially supported by the Russian Foundation for Basic Research (grant 07-01-00120) and the second by the grants of the Norwegian Research Council \#177355/V30, and of  the
European Science Foundation RNP HCAA.}



\date{May 2, 2008}

\begin{abstract}
We consider the L\"owner differential equation generating univalent maps of the unit disk (or
of the upper half-plane) onto itself minus a single slit. We prove that the circular slits,
tangent to the real axis are generated by H\"older continuous driving terms with exponent 1/3
in the L\"owner equation. Singular solutions are  described, and the critical value of the
norm of driving terms generating quasisymmetric slits in the disk is obtained.
\end{abstract}
\maketitle

\section{Introduction}

Let $\mathbb{D}=\{z\in\mathbb{C}:\,\,|z|<1\}$ be the unit disk and $\mathbb T:=\partial
\mathbb D$. The famous L\"owner equation was introduced in 1923 \cite{Loewner} in order to
represent a dense subclass of the whole class of univalent conformal maps
$f(z)=z(1+c_1z+\dots)$ in $\mathbb D$ by the limit
\[
f(z)=\lim\limits_{t\to\infty}e^tw(z,t),\quad z\in \mathbb D,
\]
where $w(z,t)=e^{-t}z(1+c_1(t)z+\dots)$ is a solution to the equation
\begin{equation}
\frac{dw}{dt}=-w\frac{e^{iu(t)}+w}{e^{iu(t)}-w}, \quad w(z,0)\equiv z, \label{LE}
\end{equation}
with a continuous driving term $u(t)$ on $t\in [0,\infty)$, see \cite[page
117]{Loewner}. All functions $w(z,t)$ map $\mathbb D$ onto $\Omega(t)\subset \mathbb D$. If
$\Omega(t)= \mathbb D\setminus \gamma(t)$, where $\gamma(t)$ is a Jordan curve in $\mathbb D$ except one of its endpoints, then the
driving term $u(t)$ is uniquely defined and we call the corresponding map $w$ a {\it slit
map}. However, from 1947 \cite{Kufarev} it is known that  solutions to (\ref{LE}) with
continuous $u(t)$ may give non-slit maps, in particular, $\Omega(t)$ can be a family of
hyperbolically convex digons in $\mathbb D$.

Marshall and Rohde \cite{Marshall} addressed the following question: {\it Under which
condition on the driving term $u(t)$ the solution to (\ref{LE}) is a slit map?} Their result
states that if $u(t)$ is Lip(1/2) (H\"older continuous with  exponent 1/2), and if for a certain
constant $C_{\mathbb D}>0$, the norm $\|u\|_{1/2}$ is bounded $\|u\|_{1/2}<C_{\mathbb D}$, then the solution $w$ is a
slit map, and moreover, the Jordan arc $\gamma(t)$ is s quasislit (a quasiconformal image of an interval within a Stolz angle). As they also proved, a
converse statement without the norm restriction holds. The absence of the norm restriction in
the latter result is essential. On one hand, Kufarev's example \cite{Kufarev} contains
$\|u\|_{1/2}=3\sqrt{2}$, which means that $C_{\mathbb D}\leq 3\sqrt{2}$. On the other hand,
Kager, Nienhuis, and Kadanoff  \cite{Kadanoff} constructed exact slit solutions to the
half-plane version of the L\"owner equation with arbitrary norms of the driving term.

Let us give here the half-plane version of the L\"owner equation. Let $\mathbb H=\{z: \im
z>0\}$, $\mathbb R=\partial \mathbb H$. The functions $h(z,t)$, normalized near infinity by
$h(z,t)=z-2t/z+b_{-2}(t)/z^2+\dots$,  solving the equation
\begin{equation}
\frac{dh}{dt}=\frac{-2}{h-\lambda(t)}, \quad h(z,0)\equiv z, \label{LE2}
\end{equation}
where $\lambda(t)$ is a real-valued
continuous driving term, map $\mathbb H$ onto a subdomain of $\mathbb H$. The question about the slit mappings and the behaviour of the driving
term $\lambda(t)$ in the case of the half-plane $\mathbb H$ was addressed by Lind \cite{Lind}.
The techniques used by Marshall and Rohde carry over to prove a similar result in the case of
the equation (\ref{LE2}), see \cite[page 765]{Marshall}. Let us denote by $C_{\mathbb H}$ the
corresponding bound for the norm $\|\lambda\|_{1/2}$. The main result by Lind is the sharp
bound, namely $C_{\mathbb H}=4$.

In some papers, e.g., \cite{Kadanoff, Lind}, the authors work with  equations (\ref{LE},
\ref{LE2}) changing (--) to (+) in their right-hand sides, and with the mappings of slit domains
onto $\mathbb D$ or $\mathbb H$. However, the results remain the same for both versions.

Marshall and Rohde \cite{Marshall} remarked that there exist many examples of driving
terms $u(t)$ which are not Lip(1/2), but which generate slit solutions with simple arcs $\gamma(t)$. In
particular, if $\gamma(t)$ is tangent to $\mathbb T$, then $u(t)$ is never Lip(1/2).

Our result states that if $\gamma(t)$ is a circular arc tangent to $\mathbb R$, then the
driving term $\lambda(t)\in$Lip(1/3). Besides, we prove that $C_{\mathbb D}=C_{\mathbb H}=4$,
and consider properties of singular solutions to the one-slit L\"owner equation.

The authors are greateful for the referee's remarks which improved the presentation.

\section{Circular tangent slits}

We shall work with the half-plane version of the L\"owner equation and with the sign (+) in
the right-hand side, consequently with the maps of slit domains onto $\mathbb H$.

We construct a mapping of the half-plane $\mathbb H$ slit along a
circular arc $\gamma(t)$ of radius 1 centered on $i$ onto $\mathbb H$ starting at the origin
directed, for example, positively. The inverse mapping we denote by $z=f(w,t)=w-2t/w+\dots$.
Then $\zeta=1/f(w,t)$ maps $\mathbb H$ onto the lower half-plane slit along a ray co-directed
with $\mathbb R^+$ and having the distance 1/2 between them. Let $\zeta_0$ be the tip of this
ray. Applying the Christoffel-Schwarz formula we find $f$ in the form
\begin{equation}
\frac{1}{f(w,t)}=\int\limits_{0}^{1/w}\frac{(1-\gamma w)\, dw}{(1-\alpha w)^2(1-\beta w)}=
\frac{\beta-\gamma}{(\alpha-\beta)^2}\log\frac{w-\alpha}{w-\beta}+\frac{\alpha-\gamma}
{\alpha-\beta}\frac{1}{w-\alpha},\label{Eq3}
\end{equation} where the branch of logarithm vanishes at infinity, and $f(w,t)$ is expanded
near infinity as
\[
f(w,t)=w-\frac{2t}{w}+\dots
\]
The latter expansion gives us two conditions: there is no constant term and the
coefficient is $-2t$ at $w$,
which implies $\gamma=2\alpha+\beta$ and $\alpha(\alpha+2\beta)=-6t$.
The condition $\im \zeta_0=-1/2$
 yields
 \[
 \frac{-2\alpha}{(\alpha-\beta)^2}=\frac{1}{2\pi}.
 \]
Then, $\beta=\alpha+2\sqrt{-\alpha\pi}$, and $\alpha(3\alpha+4\sqrt{-\alpha\pi})=-6t$.
Considering the latter equation with respect to $\alpha$ we expand the solution $\alpha(t)$
in powers of $t^{1/3}$. Hence,
 \[
 \alpha(t)=-\left(\frac{9}{4\pi}\right)^{1/3}t^{2/3}+A_2t+A_3t^{4/3}+\dots
 \]
 and
 \[
 \beta(t)=(12\pi)^{1/3}t^{1/3}+B_2t^{2/3}+\dots
 \]
Formula (\ref{Eq3}) in the expansion form regarding to $1/w$  gives
\begin{equation}
\frac{\beta-\alpha}{2\pi}\;\frac{1}{w}+\frac{\beta^2-\alpha^2}{4\pi}\;\frac{1}{w^2}+\dots+
\left(1+2\frac{\alpha}{\beta}+2\frac{\alpha^2}{\beta^2}+\dots\right)\left(\frac{1}{w}+
\frac{\alpha}{w^2}+\dots\right)=\zeta.\label{Eq4}
\end{equation}
Remember that this formula is obtained under the conditions $\gamma=2\alpha+\beta$ and $(\alpha-\beta)^2=4\alpha\pi$.
We substitute the expansions  of $\alpha(t)$ and $\beta(t)$ in this formula and consider
it as an equation for the implicit  function $w=h(z,t)$. Calculating coefficients $B_2\dots B_4$ in terms of $A_2,\dots, A_4$, and verifying $A_2=-3/4\pi$ we come to the following expansion for $h(z,t)$: $$w=h(z,t)=h(\frac{1}{\zeta},t)=\frac{1}{\zeta}
+2\zeta t+\frac{3}{2}(12\pi)^{1/3}t^{4/3}+\dots.$$ 

This version of the L\"owner equation admits the form
\begin{equation}
\frac{dh}{dt}=\frac{2}{h-\lambda(t)}, \quad h(z,0)\equiv z.\label{Eq5}
\end{equation}
Being extended onto $\mathbb R\setminus \lambda(0)$ the function $h(z,t)$ satisfies the same
equation. Let us consider $h(z,t)$, $z\in \widehat{\mathbb H}\setminus \lambda(0)$ with a singular
point at $\lambda(0)$, where $\widehat{\mathbb H}$ is the closure of $\mathbb H$. Then
\[
\lambda(t)=h(z,t)-\frac{2}{dh(z,t)/dt}=\lambda(0)+(12\pi)^{1/3}t^{1/3}+\dots
\]
about the point  $t=0$.
Thus, the driving term $\lambda(t)$ is Lip(1/3) about the point $t=0$ and analytic
for the rest of the points $t$.

\begin{remark} The radius of the circumference is not essential for the properties of $\lambda(t)$.
Passing from $h(z,t)$ to the function $\frac{1}{r}h(rz,t)$ we recalculate the coefficients
of the function $h(z,t)$ and the corresponding coefficients in the expansion of $\lambda(t)$ that depend continuously
on $r$. Therefore,
they stay within bounded intervals whenever $r$ ranges within the bounded interval.
\end{remark}

\begin{remark} In particular, the expansion for $h(z,t)$ reflects the Marshall and Rohde's remark
\cite[page 765]{Marshall} that the tangent slits can not be generated by driving terms from Lip(1/2).
\end{remark}

\section{Singular solutions for slit images}

Suppose that the L\"owner equation (\ref{Eq5}) with driving term $\lambda(t)$ generates a map
$h(z,t)$ from $\Omega(t)=\mathbb H\setminus\gamma(t)$ onto $\mathbb H$, where $\gamma(t)$ is a
quasislit. Extending $h$ to the boundary $\partial\Omega(t)$ we obtain a
correspondence between $\gamma(t)\subset\partial\Omega(t)$ and a segment $I(t)\subset\mathbb
R$, while the remaining boundary part $\mathbb R=\partial\Omega(t)\setminus\gamma(t)$
corresponds to $\mathbb R\setminus I(t)$. The latter mapping is described by solutions to the
Cauchy problem for the differential equation (\ref{Eq5}) with the initial data
$h(x,0)=x\in\mathbb R\setminus\lambda(0)$. The set $\{h(x,t): x\in\mathbb
R\setminus\lambda(0)\}$ gives $\mathbb R\setminus I(t)$, and $\lambda(t)$ does not catch
$h(x,t)$ for all $t\geq0$, see \cite{Lind} for details.

The image $I(t)$ of $\gamma(t)$ can be also described by solutions $h(\lambda(0),t)$ to
(\ref{Eq5}), but the initial data $h(\lambda(0),0)=\lambda(0)$ forces $h$ to be singular at $t=0$
and to possess the following properties.

(i) There are two singular solutions $h^-(\lambda(0),t)$ and $h^+(\lambda(0),t)$ such that
$I(t)=[h^-(\lambda(0),t),h^+(\lambda(0),t)]$.

(ii) $h^{\pm}(\lambda(0),t)$ are continuous for $t\geq0$ and have continuous derivatives for
all $t>0$.

(iii) $h^-(\lambda(0),t)$ is strictly decreasing and $h^+(\lambda(0),t)$ is strictly increasing,
so that $h^-(\lambda(0),t)<\lambda(t)<h^+(\lambda(0),y)$.

We will focus on studying the singularity character of $h^{\pm}$ at $t=0$.

\begin{theorem} Let the L\"owner differential equation (\ref{Eq5}) with the driving term
$\lambda\in\text{Lip(1/2)}$, $\|\lambda\|_{1/2}=c$, generate slit maps $h(z,t): \mathbb
H\setminus\gamma(t)\to\mathbb H$ where $\gamma(t)$ is a quasislit. Then
$h^+(\lambda(0),t)$ satisfies the condition $$\lim_{t\to 0+}\sup\frac{h^+(\lambda(0),t)-
h^+(\lambda(0),0)}{\sqrt t}\leq\frac{c+\sqrt{c^2+16}}{2},$$ and this estimate is the best possible.
\end{theorem}
\begin{proof} Assume without loss of generality that $h^+(\lambda(0),0)=\lambda(0)=0$. Denote
$\varphi(t):=h^+(\lambda(0),t)/\sqrt t$, $t>0$. This function has a continuous derivative and
satisfies the differential equation
$$t\varphi'(t)=\frac{2}{\varphi(t)-\lambda(t)/\sqrt{t}}-\frac{\varphi(t)}{2}.$$ This implies together
with property (iii) that $\varphi'(t)>0$ iff $$\frac{\lambda(t)}{\sqrt
t}<\varphi(t)<\varphi_1(t):=\frac{\lambda(t)}{2\sqrt t}+\sqrt{\frac{\lambda^2(t)}{4t}+4}.$$
Observe that $\varphi_1(t)\leq A:=(c+\sqrt{c^2+16})/2$.

Suppose that $\lim_{t\to0+}\sup\varphi(t)=B>A$, including the case $B=\infty$. Then there
exists $t^*>0$, such that $\varphi(t^*)>B-\epsilon>A$, for a certain $\epsilon>0$. If
$B=\infty$, then replace $B-\epsilon$ by $B'>A$. Therefore, $\varphi'(t^*)<0$ and $\varphi(t)$
increases as $t$ runs from $t^*$ to 0. Thus, $\varphi(t)>B-\epsilon$ for all $t\in(0,t^*)$ and
we obtain from (\ref{Eq5}) that $$\frac{dh^+(\lambda(0),t)}{dt}\leq\frac{2}{\sqrt
t(B-\epsilon-c)},$$ for such $t$. Integrating this inequality we get $$h^+(\lambda(0),t)\leq\frac{4\sqrt t
}{B-\epsilon-c}<\frac{4\sqrt t}{A-c},$$ that contradicts our supposition. This proves the
estimate of Theorem 3.1.

In order to attain the equality sign in Theorem 3.1, one chooses  $\lambda(t)=c\sqrt t$. Then
$h^+(\lambda(0),t)=A\sqrt t$ solves equation (\ref{Eq5}) with singularity at $t=0$. This
completes the proof.
\end{proof}

\begin{remark} Estimates similar to Theorem 3.1 hold for the other singular solution
$h^-(\lambda(0).t)$.
\end{remark}

\begin{remark} Let us compare Theorem 3.1 with the results from Section 2. The image
of a circular arc $\gamma(t)\subset\mathbb H$ tangent to $\mathbb R$ is
$I(t)=[h^-(\lambda(0),t),h^+(\lambda(0),t)]$, where
$h^-(\lambda(0),t)=\alpha(t)=-(9/4\pi)^{1/3}t^{2/3}+\dots$, and
$h^+(\lambda(0),t)=\beta(t)=(12\pi)^{1/3}t^{1/3}+\dots$, so that
$h^-(\lambda(0),t)\in\text{Lip}(2/3)$ and $h^+(\lambda(0),t)\in\text{Lip}(1/3)$.
\end{remark}

\begin{remark} Singular solutions to the differential equation (\ref{Eq5}) appear not only at $t=0$
but at any other moment $\tau>0$. More precisely, there exist two families $h^-(\gamma(\tau),t)$ and
$h^+(\gamma(\tau),t)$, $\tau\geq0$, $t\geq\tau$, of singular solutions to (\ref{Eq5}) that
describe the image of arcs $\gamma(t)$, $t\geq\tau$ under map $h(z,t)$. They correspond to
the initial data $h(\gamma(\tau),\tau)=\lambda(\tau)$ in (\ref{Eq5}) and satisfy the inequalities
$h^-(\gamma(\tau),t)<\lambda(t)<h^+(\gamma(\tau),t)$, $t>\tau$. These two families of singular
solutions have no common inner points and fill in the set $$\{(x,t): h^-(\lambda(0),t)\leq
x\leq h^+(\lambda(0),t), 0\leq t\leq t_0\},$$ for some $t_0$.
\end{remark}

\section{Critical norm values for driving terms}

In this section we discuss the results and techniques of Marshall and Rohde \cite{Marshall}
and Lind \cite{Lind}. The authors of \cite{Marshall} proved the existence of $C_{\mathbb D}$ such
that driving terms $u(t)\in\text{Lip}(1/2)$ with $\|u\|_{1/2}<C_{\mathbb D}$ in (\ref{LE})
generate quasisymmetric slit maps. This result remains true for an absolute number $C_{\mathbb
H}$ in the half-plane version of the L\"owner differential equation (\ref{LE2}), see e.g.
\cite{Lind}.

Lind \cite{Lind} claimed that the disk version (\ref{LE}) of the L\"owner differential equation
is `more challenging', than the half-plane version (\ref{LE2}). Working with the half-plane
version she showed that $C_{\mathbb H}=4$. The key result is based on the fact that if
$\lambda(t)\in\text{Lip}(1/2)$ in (\ref{LE2}), and $h(x,t)=\lambda(t)$, say at $t=1$, then
$\Omega(t)=h(\mathbb H,t)$ is not a slit domain and $\|\lambda\|_{1/2}\geq4$. Moreover, there
is an example of $\lambda(t)=4-4\sqrt{1-t}$ that yields $h(2,1)=\lambda(1)$. Although there
may be more obstacles for generating slit half-planes than that of the driving term $\lambda$
catching up some solution $h$ to (\ref{LE2}), Lind showed that this is basically the only
obstacle. The latter statement was proved by using techniques of \cite{Marshall}.

We will modify here the main Lind's reasonings so that they could be applied
to the disk version of the L\"owner equation. After that it remains to refer to
\cite{Marshall} and \cite{Lind} to state that $C_{\mathbb D}$ also equals 4.

Suppose that slit disks $\Omega(t)$ correspond to $u\in\text{Lip}(1/2)$ in (\ref{LE}) with the
sign `+' in its right-hand side instead of `-'. Then the maps $w(z,t)$ are extended continuously to
$\mathbb T\setminus\{e^{iu(0)}\}$. Let $z_0\in\mathbb T\setminus\{e^{iu(0)}\}$, and
let $\alpha(t,\alpha_0):=\arg w(z_0,t)$ be a solution to the following real-valued initial value
problem
\begin{equation}
\frac{d\alpha(t)}{dt}=\cot\frac{\alpha-u}{2},\;\;\;\alpha(0)=\alpha_0.\label{Eq6}
\end{equation}

Similarly, suppose that slit half-planes $\Omega(t)$ correspond to $\lambda\in\text{Lip}(1/2)$
in (\ref{LE2}) with the sign `+' in its right-hand side instead of `-'. Then the maps $h(z,t)$
are extended continuously to $\mathbb R\setminus\lambda(0)$. Let $x_0\in\mathbb
R\setminus\lambda(0)$ and let $x(t,x_0):=h(x_0,t)$ be a solution to the following real-valued
initial value problem
\begin{equation} \frac{dx(t)}{dt}=\frac{2}{x(t)-\lambda(t)},\;\;\;x(t_0)=x_0.\label{Eq7}
\end{equation}

For all $t\geq0$, $\tan((\alpha(t)-u(t))/2)\neq0$ in (\ref{Eq6}), and $x(t)-\lambda(t)\neq0$
in (\ref{Eq7}) (see \cite{Lind} for the half-plane version). Let us show a connection between the
solutions $\alpha(t)$ to (\ref{Eq6}), and $x(t)$ to (\ref{Eq7}), where the driving terms $u(t)$ and
$\lambda(t)$ correspond to each other.

\begin{lemma} Given $\lambda(t)\in\text{\rm Lip}(1/2)$, there exists $u(t)\in\text{\rm Lip}(1/2)$, such
that equations (\ref{Eq6}) and (\ref{Eq7}) have the same solutions. Conversely, given
$u(t)\in\text{\rm Lip}(1/2)$ there exists $\lambda(t)\in\text{\rm Lip}(1/2)$, such that equations
(\ref{Eq6}) and (\ref{Eq7}) have the same solutions.\end{lemma}

\begin{proof} Given $\lambda(t)\in\text{Lip}(1/2)$, denote by $x(t,x_0)$ a solution to the
initial value problem (\ref{Eq7}). Then the solution $\alpha(t,\alpha_0)$ to the initial value
problem (\ref{Eq6}) is equal to $x(t,\alpha_0)$ when
$$\tan\frac{\alpha-u}{2}=\frac{x-\lambda}{2},$$ and
$$x_0=\lambda(0)+2\tan\frac{\alpha_0-u(0)}{2}.$$ The function  $u(t)$ is normalized by choosing
$$u(0)=x_0-\arctan\frac{x_0-\lambda(0)}{2}.$$ This condition makes $\alpha_0$ and $x_0$ equal.
Hence, the first part of Lemma 1 is true if we put
\begin{equation}u(t)=x(t,x_0)-2\arctan\frac{x(t,x_0)-\lambda(t)}{2}.\label{Eq8}\end{equation}
Obviously, (\ref{Eq8}) preserves the $\text{Lip}(1/2)$ property.

Conversely, given $u(t)\in\text{Lip}(1/2)$, a solution $x(t,x_0)$ is equal to
$\alpha(t,\alpha_0)$ when
\begin{equation}\lambda(t)=\alpha(t,\alpha_0)-2\tan\frac{\alpha(t,\alpha_0)-u(t)}{2}.
\label{Eq9}\end{equation} Again (\ref{Eq9}) preserves the $\text{Lip}(1/2)$ property. This
ends the proof.
\end{proof}

Observe that in some extreme cases relations (\ref{Eq8}) or (\ref{Eq9}) preserve not only the
Lipschitz class but also its norm. Lind \cite{Lind} gave an example of the driving term
$\lambda(t)=4-4\sqrt{1-t}$ in (\ref{Eq7}). It is easily verified that $x(t,2)=4-2\sqrt{1-t}$.
If $t=1$, then $x(1,2)=\lambda(1)=4$, and $\lambda$ cannot generate slit half-plane at $t=1$.
This implies that $C_{\mathbb H}\leq4$. Going from (\ref{Eq7}) to (\ref{Eq6}) we use
(\ref{Eq8}) to put
$$u(t)=x(t,2)-2\arctan\frac{x(t,2)-\lambda(t)}{2}=4-2\sqrt{1-t}-2\arctan\sqrt{1-t}.$$ From
Lemma 4.1 we deduce that $\alpha(1,2)=u(1)$. Hence $u$ cannot generate slit disk at $t=1$, and
$C_{\mathbb D}\leq\|u\|_{1/2}$. Since $$\sup_{0\leq
t<1}\frac{u(1)-u(t)}{\sqrt{1-t}}=\sup_{0\leq
t<1}\left(2+2\frac{\arctan\sqrt{1-t}}{\sqrt{1-t}}\right)=4,$$ we have that $\|u\|_{1/2}\leq4$.
It is now an easy exercise to show that $\|u\|_{1/2}=4$. This implies that $C_{\mathbb
D}\leq4$.

\begin{lemma} Let $u\in\text{\rm Lip}(1/2)$ in (\ref{Eq6}) with $u(0)=0$ and $\alpha_0\in(0,\pi)$.
Suppose that $\alpha(t)$ is a solution to (\ref{Eq6}) and $\alpha(1)=u(1)$. Then
$\|u\|_{1/2}\geq4$.\end{lemma}

\begin{proof} Observe that $\alpha(t)$ is increasing on $[0,1]$, and $\alpha(t)-u(t)>0$ on
$(0,1)$. Let $u\in\text{Lip}(1/2)$ in (3), and $\|u\|_{1/2}=c$. Then, \begin{equation}
\alpha(t)-u(t)\leq\alpha(1)-u(1)+c\sqrt{1-t}=c\sqrt{1-t}.\label{Eq10}\end{equation} Given
$\epsilon>0$, there exists $\delta>0$, such that
$$\tan\frac{c\sqrt{1-t}}{2}<\frac{c\sqrt{1-t}}{2}(1+\epsilon),$$
for $1-\delta<t<1$ and all $0<c\leq4$. We apply this inequality to
(\ref{Eq6}) and obtain that
$$\frac{d\alpha}{dt}\geq\cot\frac{c\sqrt{1-t}}{2}>\frac{2}{c\sqrt{1-t}(1+\epsilon)}.$$
Integrating gives that $$\alpha(1)-\alpha(t)\geq\frac{4\sqrt{1-t}}{c(1+\epsilon)}.$$

This allows us to improve (\ref{Eq10}) to \begin{equation}\alpha(t)-u(t)\leq\alpha(1)-
\frac{4\sqrt{1-t}}{c(1+\epsilon)}-u(1)+c\sqrt{1-t}=\left(c-\frac{4}{c(1+\epsilon)}
\right)\sqrt{1-t}.\label{Eq11}\end{equation} 

Repeating these iterations we get
$$
\alpha(t)-u(t)\leq c_n\sqrt{1-t},
$$
where $c_0=c$, $c_{n+1}=c-4/[(1+\varepsilon)c_n]$, and $c_n>0$.
Let $g_n$ be recursively
defined by (see Lind \cite{Lind}) $$g_1(y)=y-\frac{4}{y},\;\;\;g_n(y)=y-\frac{4}{g_{n-1}(y)},
\;\;\;n\geq2.$$ It is easy to check that $c_n<g_n((1+\varepsilon)c)<(1+\varepsilon)c_n$

Lind \cite{Lind} showed that $g_n(y_n)=0$ for an increasing sequence $\{y_n\}$, and
$g_{n+1}(y)$ is an increasing function from $(y_n,\infty)$ to $\mathbb R$. So
$c(1+\epsilon)>y_n$ for all $n$, and it remains to apply Lind's result \cite{Lind} that
$\lim_{n\to\infty}y_n=4$. Hence, $c\geq4/(1+\epsilon)$. The extremal estimate is obtained if
$\epsilon\to0$ which leads to $c\geq4$. This completes the proof.
\end{proof}

Now Lind's reasonings in \cite{Lind} based on the techniques from \cite{Marshall} give a proof
of the following statement.

\begin{proposition}If $u\in\text{\rm Lip}(1/2)$ with $\|u\|_{1/2}<4$, then the domains $\Omega(t)$
generated by the L\"owner differential equation (\ref{LE}) are disks with quasislits.
\end{proposition}

In other words, Proposition 4.1 states that $C_{\mathbb D}=C_{\mathbb H}=4$.

\end{document}